\documentclass{amsart}
\usepackage{amsfonts,amssymb,amsmath}
\usepackage{url}

\newtheorem{thrm}{Theorem}[section]
\newtheorem{lem}[thrm]{Lemma}

\theoremstyle{definition}

\newcommand{\D}{{\rm d}}

\begin{document}

\title[Mean width of random polytopes]{The mean width of circumscribed random polytopes}

\author{K\'aroly J. B\"or\"oczky and Rolf Schneider}

\thanks{The first author was supported by OTKA grants 068398 and 049301 and by the EU Marie Curie project BudAlgGeo (MTKD-CT-2004-002988). Both authors were supported by EU Marie Curie projects DiscConvGeo (MTKD-CT-2005-014333) and PHD (MCRN-2004-511953).}
\keywords{Random polytope, mean width, approximation}
\subjclass[2000]{Primary 52A22, Secondary 60D05, 52A27}

\maketitle

\begin{center} {\em Dedicated to Professor Tibor Bisztriczky on the occasion of his 60th birthday} \end{center}

\begin{abstract}
For a given convex body $K$ in ${\mathbb R}^d$, a random polytope $K^{(n)}$ is defined (essentially) as the intersection of $n$ independent closed halfspaces containing $K$ and having an isotropic and (in a specified sense) uniform distribution. We prove upper and lower bounds, of optimal orders, for the difference of the mean widths of $K^{(n)}$ and $K$, as $n$ tends to infinity. For a simplicial polytope $P$, a precise asymptotic formula for the difference of the mean widths of $P^{(n)}$ and $P$ is obtained.
\end{abstract}

\section{Introduction and Results} \label{sect1}

The convex hull of $n$ independent, uniformly distributed random points in a given convex body $K$ in $d$-dimensional Euclidean space is a type of random polytope that has been studied extensively (basic references are found in the surveys \cite{WW93} and \cite{Sch04}, see also \cite{Gru97}). As in the seminal papers of R\'enyi and Sulanke \cite{RS63, RS64} (restricted to the planar case), which initiated this line of research, most of the investigations deal with asymptotic results, for $n$ tending to infinity. In a third paper, R\'enyi and Sulanke \cite{RS68} have studied a dual way of generating random polytopes related to a convex body $K$ (again in the plane), by taking intersections of independent random closed halfspaces containing the body. Subsequently, this approach has attracted less attention than the convex hulls of random points, although it deserves similar interest. In the present paper, we obtain some results on random polytopes generated in the second way.

Throughout the following, $K$ is a convex body with interior points in
$d$-dimen\-sion\-al Euclidean space ${\mathbb R}^d$ ($d \ge 2$). 
For any notions on convexity in this paper, see the monographs
Schneider \cite{Sch93} or  Gruber \cite{Gru07}. 
Let $B^d$ be the unit ball of ${\mathbb R}^d$ with center at the origin; 
then $K_1:=K+B^d$ is the parallel body of $K$ at distance $1$. By ${\mathcal H}$ we denote the space (with its usual topology) of hyperplanes in ${\mathbb R}^d$, and ${\mathcal H}_K$ is the subspace of hyperplanes meeting $K_1$ but not the interior of $K$. For $H\in {\mathcal H}_K$, the closed halfspace bounded by $H$ that contains $K$ is denoted by $H^-$. The measure $\mu$ is the motion invariant Borel measure on ${\mathcal H}$, normalized so that $\mu(\{H\in {\mathcal H}: H\cap M\not=\emptyset\})$ is the mean width $W(M)$ of $M$, for every convex body $M\subset {\mathbb R}^d$. Let $2\mu_K$ be the  restriction of $\mu$ to ${\mathcal H}_K$. Since $\mu({\mathcal H}_K)= W(K+B^d)-W(K)= W(B^d)=2$, the measure $\mu_K$ is a probability measure. For $n\in{\mathbb N}$, let $H_1,\dots,H_n$ be independent random hyperplanes in ${\mathbb R}^d$ (${\mathcal H}$-valued random variables on some probability space $(\Omega,{\bf A},{\mathbb P})$), each with distribution $\mu_K$. The intersection $\bigcap_{i=1}^n H_i^-$ is a random polyhedral set, possibly unbounded. We put
\[ K^{(n)}:= \bigcap_{i=1}^n H_i^-\cap K_1\]
and ask for ${\mathbb E}W(K^{(n)})$, where ${\mathbb E}$ denotes mathematical expectation. Alternatively, we might consider ${\mathbb E}_1 W(K^{(n)})$, the conditional expectation of $W(K^{(n)})$ under the condition that $\bigcap_{i=1}^nH_i^-\subset K_1$. Since ${\mathbb E}W(K^{(n)})={\mathbb E}_1 W(K^{(n)})+O(\gamma^n)$ with $\gamma \in (0,1)$, as is easy to see, there is no difference in the asymptotic behaviors of both quantities, as $n\to\infty$. We also remark that, for the asymptotic results, the parallel body $K_1$ could be replaced by any other convex body containing $K$ in its interior; this would only affect some normalization constants. 

The preceding model has to be distinguished from the one where $n$ independent random points are chosen from the boundary of $K$, and the intersection of the supporting halfspaces of $K$ at these points is the random polyhedron under consideration. For this model and sufficiently smooth convex bodies,  B\"or\"oczky and Reitzner \cite{BR04} have derived asymptotic expansions of the expectations of volume, surface area and mean width.

For comparison, we mention first some results involving convex hulls of random points. Let $K_n$ be the convex hull of $n$ independent, uniformly distributed random points in the convex body $K$.  Throughout this paper, $c_1,c_2,\dots$ are positive constants that depend only on $K$ and $d$. In writing $n\ge n_0$, we indicate that a result is true for all sufficiently large $n\in {\mathbb N}$.
There exist $c_1,\dots,c_4$ such that, for $n\ge n_0$,
\begin{equation}\label{1}
c_1 n^{-2/(d+1)} < W(K)-{\mathbb E}W(K_n) < c_2n^{-1/d},
\end{equation}
\begin{equation}\label{2}
c_3 n^{-1}\ln ^{d-1}n < V(K)-{\mathbb E}V(K_n) < c_4n^{-2/(d+1)},
\end{equation}
where $V$ denotes the volume.
The inequalities (\ref{1}) are due to Schneider \cite{Sch87}, and (\ref{2}) to  B\'{a}r\'{a}ny and Larman \cite{BL88}. The orders are best possible, being attained in (\ref{1})(left) and (\ref{2})(right) by sufficiently smooth bodies, and in (\ref{1})(right) and (\ref{2})(left) by polytopes. 

If one sets about obtaining analogous results for random polytopes obtained as intersections of halfspaces, the idea of dualizing immediately comes to mind. Supposing that $o\in{\rm int}\,K$, polarization with respect to the unit sphere sends $K$ to its polar body $K^*$, and it interchanges hyperplanes not meeting ${\rm int}\,K$ with points in $K^*$ (hence, mean width and volume should interchange their roles); {\it cum grano salis}, intersections of halfspaces correspond to convex hulls of points. Measures on hyperplanes correspond to measures on points; however, uniform measures do not correspond to uniform measures, and mean width and volume are not exactly related under polarity. Nevertheless, this approach can be put to work in some cases, or arguments may find their dual analogs in a more heuristic way. In this more or less vague sense, duality has been applied to polygons and vertex numbers in the plane by Ziezold \cite{Zie70}, and to smooth and general convex bodies in higher dimensions by Kaltenbach \cite{Kal90}. In particular, Kaltenbach has established a counterpart to (\ref{1}), namely
\begin{equation}\label{4}
c_5 n^{-2/(d+1)} < {\mathbb E}V(K^{(n)})-V(K) < c_6n^{-1/d}.
\end{equation}
We obtain here a counterpart to (\ref{2}), again with optimal orders.

\begin{thrm}\label{t1}
For $n\ge n_0$,
\begin{equation}\label{5}
c_7 n^{-1}\ln^{d-1}n < {\mathbb E}W(K^{(n)})-W(K) < c_8n^{-2/(d+1)}.
\end{equation}
\end{thrm}

The right side follows from a result of independent interest.

\begin{thrm}\label{t2}
For each $n\in{\mathbb N}$, the functional $K\mapsto {\mathbb E}W(K^{(n)})/W(K)$ attains its maximum at balls.
\end{thrm}

The following precise asymptotic formula is a counterpart to a result of Affentranger and Wieacker \cite{AW91}.
\begin{thrm}\label{t3}
If $P$ is a simplicial polytope in ${\mathbb R}^d$ with $r$ facets, then, as $n\to\infty$,
\begin{equation}
{\mathbb E}W(P^{(n)}) - W(P) \sim \frac{2rd}{(d+1)^{d-1}}\,\frac{\ln^{d-1}n}{n}.
\end{equation}
\end{thrm}

For the polytope $P$ in Theorem \ref{t3}, we also obtain asymptotic results for the numbers of vertices and facets. We denote by $f_k(P^{(n)})$ the number of $k$-faces of $P^{(n)}$ that are contained in the interior of $P_1$ (recall that ${\mathbb P}(P^{n)}\not\subset P_1) = O(\gamma^n)$ with $0<\gamma<1$). Then
\begin{equation} \label{f0}
{\mathbb E}f_0(P^{(n)})\sim\frac{rd^d}{d!}M_1(\Delta_{d-1})\ln^{d-1}n,
\end{equation}
where the constant $M_1(\Delta_{d-1})$ is given by (\ref{M}), and
\begin{equation} \label{fd-1}
{\mathbb E}f_{d-1}(P^{(n)})\sim\frac{rd}{(d+1)^{d-1}}\ln^{d-1}n.
\end{equation}

\vspace{2mm}

\noindent{\bf Remark.} Based on the paper of B\'ar\'any and Buchta \cite{BB93}
generalizing the results of Affentranger and Wieacker \cite{AW91},
one can most probably extend Theorem~\ref{t3} as follows.
Let $P$ be any polytope in ${\mathbb R}^d$.
We write $T(P)$ to denote the number of flags (or towers) of $P$; namely,
the number of chains $F_0\subset\ldots\subset F_{d-1}$
where $F_i$ is an $i$-face of $P$.
Then, as $n\to\infty$,
\begin{eqnarray*}
{\mathbb E}W(P^{(n)}) - W(P) &\sim&
\frac{2T(P)}{(d+1)^{d-1}(d-1)!}\,\frac{\ln^{d-1}n}{n},\\
{\mathbb E} f_{d-1}P^{(n)} &\sim&
\frac{T(P)}{(d+1)^{d-1}(d-1)!}\,\ln^{d-1}n,\\
{\mathbb E} f_{0}P^{(n)} &\sim&
\frac{T(P)d^d}{(d!)^2}\,M_1(\Delta_{d-1})\ln^{d-1}n.
\end{eqnarray*}
Moreover, it is well known that $M_1(\Delta_1)=\frac13$ and $M_1(\Delta_2)=\frac1{12}$, and Buchta and Reitzner \cite{BR01} proved $M_1(\Delta_3) = \frac{13}{720}-\frac{\pi^2}{15015}$.

\section{Proofs of Theorems \ref{t1}, \ref{t2}}
\label{sec2}

First, we fix some more notation. In the following, $S^{d-1}:=\{x\in{\mathbb R}^d:\langle x,x\rangle=1\}$ (´where $\langle\cdot,\cdot\rangle$ denotes the scalar product) is the unit sphere of ${\mathbb R}^d$, $\lambda$ is Lebesgue measure on ${\mathbb R}^d$, and $\sigma$ is the spherical Lebesgue measure on $S^{d-1}$. 

For sets $A_1,\dots,A_m$ and points $x_1,\dots,x_k$ in ${\mathbb R}^d$, $m,k\in{\mathbb N}_0$,  we write
$$ [A_1,\dots,A_m,x_1,\dots,x_k]:={\rm conv}(A_1\cup\ldots\cup A_m\cup\{x_1,\dots,x_k\}).$$

Before the proofs, we want to substantiate the remark, made in the introduction, about the comparison between ${\mathbb E}W(K^{(n)})$ and the conditional expectation ${\mathbb E}_1W(K^{(n)})$. Clearly, there are finitely many hyperplanes $E_j\in{\mathcal H}_K$, $j=1,\dots,k$, such that $\bigcap_{j=1}^k E^-_j\subset K+(1/2)B^d$. We can choose neighborhoods $N_j$ of $E_j$, $j=1,\dots,k$, of equal measure $\mu_K(N_j)=:\alpha\in(0,1)$, such that any hyperplanes $H_1,\dots,H_k$ with $H_j\in N_j$, $j=1,\dots,k$, satisfy $\bigcap_{j=1}^k H^-_j\subset K+B^d$. Now let $H_1,\dots,H_n$ be independent random hyperplanes with distribution $\mu_K$, and define $\bigcap_{i=1}^nH_i^-:=P_n$. The event $P_n \not\subset K_1$ occurs only if one of the events $A_j:=\{H_i\notin N_j \mbox{ for } i=1,\dots,k\}$ occurs. It follows that ${\mathbb P}(P_n\not\subset K_1) \le {\mathbb P}(\bigcup_{j=1}^kA_j)\le \sum_{j=1}^k{\mathbb P}(A_j) =k(1-\alpha)^n$. From 
\begin{align*} 
&  {\mathbb E}W(K^{(n)}) - (1-{\mathbb P}(P_n\not\subset K_1)){\mathbb E}_1W(K^{(n)})\\ 
&  = \int_{P_n\not\subset K_1} W(K^{(n)})\,\D {\mathbb P}\le W(K_1){\mathbb P}(P_n\not\subset K_1)
\end{align*}
we now conclude that ${\mathbb E}W(K^{(n)}) - {\mathbb E}_1W(K^{(n)}) = O(\gamma^n)$ with $\gamma\in(0,1)$.

Turning to the proof of Theorem \ref{t1}, we first recall that B\'{a}r\'{a}ny and Larman \cite{BL88} proved (\ref{2}) after establishing the following general result. For $x\in K$, let $v(x)$ be the minimal volume that a closed halfspace with $x$ in its boundary cuts off from $K$. For (small) $t>0$, let
\begin{equation} \label{6}
K(t) :=\{ x\in K: v(x)\le t\}.
\end{equation}

In \cite{BL88}, the existence of positive constants $c_9,c_{10}$ with
\begin{equation} \label{6a}
c_9 V(K(1/n)) < V(K) -{\mathbb E}V(K_n) < c_{10}V(K(1/n))
\end{equation}
for $n\ge n_0$ was proved. Part of this approach will now be `dualized'. 

For $x\in{\mathbb R}^d$, let $K_x:=[K,x]$, and let $w(x)$ be the $\mu$-measure of the set of hyperplanes separating $K$ and $x$, thus $w(x)=W(K_x)-W(K)$. For a hyperplane $H\in{\mathcal H}_K$, let 
$$ m(H):= \min\{w(x): x\in H\},$$
and for $t>0$ (sufficiently small), define
\[ {\mathcal H}_K(t) := \{H\in {\mathcal H}_K: m(H)\le t\}.\]
It is convenient to describe this set of hyperplanes in a different way. For this, put
\[ K[t]:=\{ x\in{\mathbb R}^d:w(x)\le t\}.\]
Let $x,y\in K[t]$, $\lambda\in[0,1]$, and $z\in K_{(1-\lambda)x+\lambda y}$. Then $z= (1-\alpha)[(1-\lambda)x +\lambda y]+\alpha k$ with suitable $k\in K$ and $\alpha\in[0,1]$. It follows that 
\[ z= (1-\lambda)[(1-\alpha)x+\alpha k] +\lambda[(1-\alpha)y+\alpha k] \in (1-\lambda)K_x+\lambda K_y,\]
thus $K_{(1-\lambda)x+\lambda y}\subset (1-\lambda)K_x+\lambda K_y$. This gives
\begin{align*}
W(K_{(1-\lambda)x+\lambda y}) &\le W((1-\lambda)K_x+\lambda K_y)=(1-\lambda)W(K_x)+\lambda W(K_y)\\
&\le (1-\lambda)(W(K)+t)+\lambda(W(K)+t) = W(K)+t,
\end{align*}
hence $w((1-\lambda)x+\lambda y)\le t$ and thus $(1-\lambda)x+\lambda y\in K[t]$. This shows that $K[t]$ is convex.

Now let $H\in{\mathcal H}_K$. If $H\cap K[t]\not=\emptyset$, then $H$ contains a point $x$ with $w(x)\le t$, hence $m(H)\le t$ and, therefore, $H\in{\mathcal H}_K(t)$. If $H\cap K[t]=\emptyset$, then every $x\in H$ satisfies $w(x)>t$, and since $m(H)$ is an attained minimum, also $m(H)>t$ and hence $H\notin{\mathcal H}_K(t)$. Thus, ${\mathcal H}_K(t)$ is the set of hyperplanes meeting the convex body $K[t]$ but not the interior of $K$. In particular, 
\[ \mu({\mathcal H}_K(t))=W(K[t]) - W(K).\]

The left inequality in (\ref{6a}) admits a straightforward dualization, as already noted by Kaltenbach \cite{Kal90}. The following argument, which we give for the reader's convenience, is the exact dual analog of that in \cite[p. 283]{BL88}. Let $H_1,\dots,H_n$ and $K^{(n)}$ be as in the introduction, $n>d$. Let $H\in {\mathcal H}_K$ and choose $x_0\in H$ such that $w(x_0)=m(H)$. If $H_i$ does not separate $x_0$ and $K$, for $i=1,\dots,n$, then $H\cap K^{(n)}\not=\emptyset$, hence
\[ {\mathbb P}(H\cap  K^{(n)}\not=\emptyset) \ge (1-m(H))^n.\]
For small $t>0$, we obtain
\begin{align*}
& {\mathbb E}W(K^{(n)})-W(K) = \int\int_{{\mathcal H}_K}{\bf 1}\{H\cap K^{(n)}\not=\emptyset\}\,\mu(\D H)\,\D{\mathbb P}\\ 
& = \int_{{\mathcal H}_K}{\mathbb P}(H\cap K^{(n)}\not=\emptyset)\,\mu(\D H)\ge \int_{{\mathcal H}_K}(1-m(H))^n\mu(\D H)\\
& > \int_{{\mathcal H}_K} {\bf 1}\{m(H)\le t\}(1-t)^n\mu(\D H) \\
&  = (1-t)^n\mu({\mathcal H}_K(t)) =(1-t)^n(W(K[t])-W(K)).
\end{align*}
The choice $t=1/n$ gives
\begin{equation}\label{7}
c_{11}(W(K[1/n])-W(K))< {\mathbb E}W(K^{(n)})-W(K).
\end{equation}

Next, we carry over results from \cite{BL88} by applying them to the polar body.  We assume that $o\in{\rm int}\,K$ and let $K^*$ denote the polar body of $K$. We write the points of ${\mathbb R}^d$ in the form $ru$ with $u\in S^{d-1}$ and $r\ge 0$ and the hyperplanes of ${\mathbb R}^d$ in the form $H(u,t) := \{x\in {\mathbb R}^d:\langle x,u\rangle =t\}$ with $u\in S^{d-1}$ and $t\ge 0$. The map $\varphi:{\mathbb R}^d\setminus \{o\}\to {\mathcal H}$ is defined by 
\begin{equation}\label{map}
\varphi(ru):= H(u,r^{-1}).
\end{equation} Let $\nu$ denote the image measure of $\lambda$ under $\varphi$, then
\[ \nu(A) = \int_{S^{d-1}}\int_0^{\infty} {\bf 1}\{H(u,t)\in A\}t^{-(d+1)}\,\D t\,\sigma(\D u)\]
for Borel sets $A\subset {\mathcal H}$. For comparison, the invariant measure $\mu$ is given by
\begin{equation}\label{8b} 
\mu(A) = \frac{2}{\sigma(S^{d-1})}\int_{S^{d-1}}\int_0^{\infty} {\bf 1}\{H(u,t)\in A\}\,\D t\,\sigma(\D u).
\end{equation}
Consequently, there exist positive constants $c_{12},c_{13}$ such that
\[ c_{12}\nu(A) \le \mu(A) \le c_{13}\nu(A) \qquad\mbox{if } A\subset {\mathcal H}_K.\]

In the following, we assume that $0\le t<t_0$, where $t_0$ is chosen such that $K_1^*\cap K^*(t_0) = \emptyset$; here $K_1^*:=(K_1)^*$, and  $K^*(t_0)$ is defined by (\ref{6}). Let $x\in K^*(t)$. There is a hyperplane $E$ through $x$ such that $\lambda(K^*\cap E^+)\le t$, where $E^+$ is the halfspace bounded by $E$ that does not contain $o$. Let $H:=\varphi(x)$ and $y:=\varphi^{-1}(E)$, then $y\in H$. The mapping $\varphi$ maps the cap $K^*\cap E^+$ bijectively onto the set of hyperplanes separating $y$ and $K$, which is denoted by ${\mathcal H}_K^y$ and is a subset of ${\mathcal H}_K$, by the choice of $t_0$. We conclude that
\[ m(H) \le \mu({\mathcal H}_K^y) \le c_{13}\nu({\mathcal H}_K^y)= c_{13}\lambda(K^*\cap E^+) \le c_{13}t\]
and hence that $H\in {\mathcal H}_K(c_{13}t)$. Since $x\in K^*(t)$ was arbitrary, this shows that $\varphi(K^*(t))\subset {\mathcal H}_K(c_{13}t)$; therefore,
\[ \lambda(K^*(t))=\nu(\varphi(K^*(t)) \le \nu({\mathcal H}_K(c_{13}t)) \le c_{12}^{-1}\mu({\mathcal H}_K(c_{13}t)).\]
Now (\ref{7}) together with this inequality gives, for $n\ge n_0$,
\[ {\mathbb E}W(K^{(n)})-W(K) \ge c_{11}\mu({\mathcal H}_K(1/n)) \ge c_{11}c_{12}\lambda(K^*(1/c_{13}n)).\]
By \cite[Th. 2]{BL88},
\[ \lambda(K^*(\epsilon))\ge c_{14}\epsilon \ln^{d-1}(1/\epsilon) \]
for $\epsilon>0$. This yields the left inequality of (\ref{5}).

The right inequality of (\ref{6a}) relies heavily on the technique of Macbeath regions (see B\'{a}r\'{a}ny \cite{Bar99, Bar06} for expositions of this technique and its applications), which does not dualize in an obvious way. The proof of the right inequality of (\ref{5}) can, however, be deduced from Theorem \ref{t2}. The latter is a counterpart to Groemer's \cite{Gro74} inequality, which says that ${\mathbb E}V(K_n)/V(K)$ is minimal if $K$ is an ellipsoid. In the subsequent proof of Theorem \ref{t2}, `dualization' becomes a bit vague: Steiner symmetrization, which is a tool in Groemer's proof, is replaced by Minkowski symmetrization.

Let $h(K,\cdot)$ be the support function of $K$. By (\ref{8b}) and the definition of the measure $\mu_K$, we have
\[ \mu_K(A) =\frac{1}{\sigma(S^{d-1})} \int_{S^{d-1}} \int_0^1 {\bf 1}\{H(u,h(K,u)+t)\in A\}\,\D t\,\sigma(\D u) \]
for Borel sets $A\subset {\mathcal H}_K$. We write $H^-(u,t):=\{x\in{\mathbb R}^d:\langle x,u\rangle\le t\}$ and use the abbreviations $U:=(u_1,\dots,u_n)$, $T:=(t_1,\dots,t_n)$ and
\[ P(K,U,T):= H^-(u_1,h(K,u_1)+t_1) \cap \ldots\cap H^-(u_n,h(K,u_n)+t_n) \cap K_1.\]
Then we get
\begin{align*}
{\mathbb E}W(K^{(n)})
& =  \int\dots\int W(H_1^-\cap\ldots\cap H_n^-\cap K_1)\,\mu_K(\D H_1)\cdots\mu_K(\D H_n)\\
& = \left(\frac{1}{\sigma(S^{d-1})}\right)^n\int_{(S^{d-1})^n} \int_{[0,1]^n}
W(P(K,U,T))\, \D T\,\sigma^n(\D U).
\end{align*}

Let $K,M\subset {\mathbb R}^d$ be two convex bodies. Let $\alpha\in [0,1]$ and $x\in (1-\alpha)P(K,U,T)+\alpha P(M,U,T)$. Then 
$x=(1-\alpha)y+\alpha z$ with $y\in P(K,U,T)$ and $z\in P(M,U,T)$. For each $i\in \{1,\dots,n\}$, we have $\langle y,u_i\rangle \le h(K,u_i)+t_i$ and $\langle z,u_i\rangle \le h(M,u_i)+t_i$, hence
$$
\langle x,u_i \rangle \le (1-\alpha)(h(K,u_i)+t_i)+\alpha(h(M,u_i)+t_i)
= h((1-\alpha)K+\alpha M,u_i)+t_i.
$$
Since also $x\in K_1$, we see that $x\in P((1-\alpha)K+\alpha M,U,T)$. This shows that
\[ (1-\alpha)P(K,U,T) +\alpha P(M,U,T) \subset P((1-\alpha)K+\alpha M,U,T)\]
and hence that
\[ W(P((1-\alpha)K+\alpha M,U,T)) \ge (1-\alpha)W(P(K,U,T)) +\alpha W(P(M,U,T)).\]
Inserting this in the representation of ${\mathbb E}W(K^{(n)})$, we obtain
\[ {\mathbb E}W([(1-\alpha)K+\alpha M]^{(n)}) \ge (1-\alpha){\mathbb E}W(K^{(n)})  +\alpha{\mathbb E}W(M^{(n)}).\]
Thus, the function $K\mapsto {\mathbb E}W(K^{(n)})$ is concave with respect to Minkowski addition, and it is clearly invariant under rigid motions and continuous with respect to the Hausdorff metric. Now the following standard argument shows that on the set of convex bodies of given mean width, the function ${\mathbb E}W(K^{(n)})$ attains its maximum at the balls. A {\em rotation mean} of $K$ is every convex body of the form $K'=m^{-1}(\delta_1 K+\ldots+\delta_m K)$ with  $m\in{\mathbb N}$ and rotations $\delta_1,\dots,\delta_m$ of ${\mathbb R}^d$. By the concavity shown above and the linearity of the mean width, we have ${\mathbb E}W((K')^{(n)})\ge {\mathbb E}W(K^{(n)})$ and $W(K')=W(K)$. By a theorem of Hadwiger (see \cite[Theorem 3.3.2]{Sch93}), there is a sequence of rotation means of $K$ converging to a ball $B$. This ball satisfies ${\mathbb E}W(B^{(n)})\ge{\mathbb E}W(K^{(n)})$ and $W(B)=W(K)$. We can write the result as
\begin{equation}\label{9}
{\mathbb E}W(K^{(n)})-W(K)\le\frac{W(K)}{2}[{\mathbb E}W((B^d)^{(n)})-W(B^d)],
\end{equation}
which proves Theorem \ref{t2}. The right side is of order $n^{-2/(d+1)}$ as $n\to\infty$. This can be deduced from (\ref{polar}) below for $K=B^d$, once the analogous result for the convex hull of independent, identically distributed points in the ball $B^d$ is known, for the case where the Lebesgue measure (yielding the distribution of the points and the volume functional) is replaced by Lebesgue measure with a density that is continuous in a neighborhood of ${\rm bd}B^d$ and constant on ${\rm bd}B^d$. Such a result, in turn, is obtained by a straightforward extension of the Lebesgue measure case, first treated by Wieacker \cite{Wie78} and generalized by Affentranger \cite{Aff91}. This completes the proof of Theorem \ref{t1}.

\section{Polarity and a useful functional}
\label{sec3}

In our preparations for the proof of Theorem \ref{t3}, we make use of the mapping $\varphi$ defined by (\ref{map}). We assume that $K$ is a convex body containing the origin $o$ in its interior. The same holds then for its polar body $K^*$. We define
$$ X_K := {\rm cl}\,(K^*\setminus K_1^*),$$
thus $\varphi(X_K)={\mathcal H}_K$. Writing $\mu^*,\mu_K^*$ for  the image measures of $\mu, \mu_K$, respectively,  under $\varphi^{-1}$, we have
$$
\mu_K^*(A)=\frac1{\sigma(S^{d-1})}\int_A\|x\|^{-(d+1)}\,\D x 
$$
for any Borel set $A\subset X_K$. For a convex body $M$ containing $K$ it is clear (and well known, see Glasauer and Gruber \cite{GlG97}) that
$$
W(M)-W(K)=\frac2{\sigma(S^{d-1})}\int_{K^*\setminus M^*}\|x\|^{-(d+1)}\,\D x.
$$

By $K_n^*$ we denote the convex hull of $K_1^*$ and $n$ independent random points in $X_K$ with distribution $\mu_K^*$. Thus, $K_n^*$ is stochastically equivalent to the polar body of $K^{(n)}$, and we have
\begin{equation}
\label{polar}
\mathbb{E}W(K^{(n)})-W(K)=\frac2{\sigma(S^{d-1})}
\mathbb{E}\int_{K^*\setminus K_n^*}\|x\|^{-(d+1)}\D x=2\mathbb{E}\,\mu^*_K(K^*\setminus K_n^*).
\end{equation}

For $H_1,\ldots,H_n\in{\mathcal H}_K$, we say that a $(d-1)$-dimensional convex compact set $F$ is a {\em proper facet} of $K^{(n)}= \bigcap_{i=1}^n H_i^-\cap K_1$, if $F=H_i\cap K^{(n)}$ for some $H_i$ which is a supporting hyperplane of $K^{(n)}$ intersecting ${\rm int}\,K_1$. Further, $v$ is a {\em proper vertex} of $K^{(n)}$ if $v\in{\rm int}\,K_1$ and $\{v\}$ is the intersection of the proper facets of $K^{(n)}$ containing $v$. We write $f_0(K^{(n)})$ and $f_{d-1}(K^{(n)})$ for the number of proper vertices and facets, respectively, of $K^{(n)}$. 

Next, let $K^*_n=[K_1^*,x_1,\ldots,x_n]$ for $x_1,\ldots,x_n\in X_K$. We say that a $(d-1)$-dimensional convex compact set $F$ is a {\em proper facet} of $K^*_n$ if $K_1^* \cap {\rm aff}\,F =\emptyset$ and $F$ is the intersection of $K^*_n$ and a supporting hyperplane of $K^*_n$. Further, some $x_i$ is a {\em proper vertex} of $K^*_n$ if $x_i\notin K_1^*$ and $\{x_i\}$ is the intersection of $K^*_n$ and a supporting hyperplane of $K^*_n$. We write $f_0(K^*_n)$ and $f_{d-1}(K^*_n)$ to denote the number of proper vertices and facets, respectively, of $K^*_n$. If $K^{(n)} =\bigcap_{i=1}^n \varphi(x_i)^-\cap K_1$, then $\varphi$ defines bijective correspondences between the proper vertices of $K^*_n$
and proper facets of $K^{(n)}$, and between the proper facets of $K^*_n$
and proper vertices of $K^{(n)}$.

As we have seen in Section \ref{sec2}, there exists a number $\gamma\in(0,1)$ depending only on $K$ such that with
probability at least $1-O(\gamma^n)$  we have $K_1^*\subset {\rm int}\,K^*_n$. In this case, $K^*_n$ is
a polytope with vertices among the $n$ random points determining $K^*_n$.

Now we assume that $P$ a simplicial polytope (with $o$ in its interior), then $P^*$ is a simple polytope. Similarly as in Affentranger and Wieacker \cite{AW91}, we consider the function $T^*_q(P^*_n)$ below, where $q\geq 0$ and
$P^*_n =[P_1^*,x_1,\ldots,x_n]$ for $x_1,\ldots,x_n\in X_P$.

If $F$ is a $(d-1)$-dimensional convex set whose 
affine hull intersects $P^*$ and avoids $P_1^*$ then let $v_F$ be a vertex of $P^*$ 
which is separated from $P_1^*$ by ${\rm aff}\,F$, and where there exists
a supporting hyperplane to $P^*$ parallel to ${\rm aff}\,F$. Further, let $S_F=[F,v_F]$. We write ${\mathcal F}(P^*_n)$ to denote the family of proper facets of $P^*_n$, and we define
$$
T^*_q(P^*_n):=\sum_{F\in {\mathcal F}(P^*_n)}\mu^*_P(S_F)^q.
$$
The functionals $T^*_q(P^*_n)$ are closely related to our problem because  $f_{d-1}(P^*_n)=T_0(P^*_n)$, and we will prove  $\mathbb{E}\,\mu^*_P(P^*\setminus P^*_n)\sim\mathbb{E}\,T^*_{1}(P^*_n)$ in Section \ref{sec4}.
Now the core of the arguments leading to Theorem \ref{t3} is the following lemma. For this, we need some notation.
The $(d-1)$-dimensional Lebesgue measure is denoted by $\lambda_{d-1}$. 
For $q\geq 0$ and a $(d-1)$-dimensional compact convex set $A$, let
\begin{equation}\label{M}
M_q(A)=\lambda_{d-1}^{-d-q}(A)\int_{A^d}
\lambda_{d-1}([x_1,\ldots,x_d])^{q}\,\lambda_{d-1}^{d}(\D(x_1,\dots,x_d)).
\end{equation}
Let $\Delta_{d-1}$ be a fixed $(d-1)$-dimensional simplex, then
$M_q(A)=M_q(\Delta_{d-1})$ for any $(d-1)$-dimensional simplex $A$,
by affine invariance. For an arbitrary $(d-1)$-dimensional compact convex set $A$,
there exists a $(d-1)$-dimensional simplex $B\subset A$ such that
$A$ is contained in a translate of $-(d-1)B$, therefore
\begin{equation}
\label{Mqest}
M_q(A)\leq (d-1)^{(d-1)(d+q)}M_q(\Delta_{d-1}).
\end{equation}

\begin{lem}\label{lem1}
If $q\geq 0$ is an integer and $P$ is a simplicial polytope in ${\mathbb R}^d$ with $r$ facets, then, as $n$ tends to infinity, 
$$
\mathbb{E}\,T^*_q(P^*_n)\sim\frac{r(d+q-1)!d^{d-1}M_{q+1}(\Delta_{d-1})}{(d-1)!^2}\,\frac{\ln^{d-1}n}{n^q}. 
$$
\end{lem}

\begin{proof}
To prove Lemma \ref{lem1}, it is sufficient to verify for any $\varepsilon>0$ the existence of $n_0$ and $\Gamma$ depending on $\varepsilon$, $q$ and $P$ such that, if $n>n_0$, then
\begin{align}\label{lower}
\mathbb{E}\,T^*_q(P^*_n) & > \frac{1}{(1+\varepsilon)^{2d+2q}}\frac{r(d+q-1)!d^{d-1}M_{q+1} (\Delta_{d-1})} { (d-1)!^2}\, \frac{\ln^{d-1}n}{n^q} -\frac{\Gamma\ln^{d-2}n}{n^q},\\
\label{upper}
\mathbb{E}\,T^*_q(P^*_n) &< (1+\varepsilon)^{2d+2q} \frac{r(d+q-1)! d^{d-1}M_{q+1} \Delta_{d-1})}{(d-1)!^2}\, \frac{\ln^{d-1}n}{n^q} +\frac{\Gamma\ln^{d-2}n}{n^q}.
\end{align}
In the rest of this section, we write $\Gamma_1,\Gamma_2,\ldots$ to denote constants that may depend on $\varepsilon$, $q$ and $P$.

Many calculations are simpler if we do them with respect to an orthonormal basis, therefore we introduce some notation.
Let $e_1,\ldots,e_d$ be an orthonormal basis of ${\mathbb R}^d$, and let $\widetilde{\Omega}=[o,e_1,\ldots,e_d]$. For $p=0,\ldots,d$, and a $(d-1)$-dimensional convex set $F$, we define
$$
\tilde{\theta}^p_{F}=\left\{
\begin{array}{ll}
1&\mbox{ if aff\,$F$ intersects each open ray ${\mathbb R}_+e_i$, $i=1,\ldots,d$,} \\
&\mbox{ and aff$\,F$ separates $p$ points out of  $e_{1},\ldots,e_{d}$ from $o$},\\
0&\mbox{ otherwise. }
\end{array} \right.
$$
In addition, we define $\widetilde{S}_F=[o,F]$, and $\tilde{\eta}_F$ denotes the distance of ${\rm aff}\,F$ from $o$. Moreover, let $\widetilde{C}_F$ be the simplex cut off by ${\rm aff}\,F$ from $\sum_{i=1}^d{\mathbb R}_{\geq 0}e_i$ if ${\rm aff}\,F$ intersects each open ray ${\mathbb R}_+e_i$, $i=1,\ldots,d$, and let $\widetilde{C}_F = \widetilde {\Omega}$ otherwise. For $s=(s_1,\ldots,s_d)\in{\mathbb R}_+^d$, we write $H(s)$ for the hyperplane $H$ that contains the points $s_ie_i$ for $i=1,\ldots,d$. It follows that (see also \cite[p. 298--299]{AW91})
\begin{equation} \label{volwiths}
V(\widetilde{C}_{H(s)})=s_1\cdots s_d/d!,
\end{equation}
\begin{equation} \label{changofvar}
\D \mu(H(s))=\tilde{\eta}_{H(s)}^{d+1}(s_1\cdots s_d)^{-2}\D s.
\end{equation}
Finally, we recall the lemma in  \cite[p. 296]{AW91}. It says that for integers $k,m\geq 0$ and $p\geq 2$, and for $c \in (0,1]$, we have
\begin{align}
\nonumber
&  \int_{(0,1]^p}(s_1\cdots s_p)^k \left(1-c\,s_1\cdots s_p\right)^{n-m} \,\D (s_1,\dots s_p)\\ 
&  = \frac{k!}{(p-1)!c^{k+1}}\,\frac{\ln^{p-1}n}{n^{k+1}}
\label{integral}
+O\left( \frac{\ln^{p-2}n}{n^{k+1}}\right)
\end{align}
as $n$ tends to infinity, where the implied constant in $O(\cdot)$ depends on $k,m,p,c$. In addition, if $p=1$, then
\begin{equation}
\label{integral1}
\int_0^1s^k(1-c\,s)^{n-m}\,\D s=\frac{k!}{c^{k+1}}\,\frac{1}{n^{k+1}}
+O\left( \frac{1}{n^{k+2}}\right).
\end{equation}

We choose a number $\omega>0$ with the following four properties:
\begin{itemize}
\item Any edge of $P^*$ is of length at least $3\omega$.
\item If $y$ is a vertex of $P^*$ and $w_{y,1},\ldots,w_{y,d}$ are
the points on the $d$ edges of $P^*$ meeting at $y$ such that
$\|w_{y,i}-y\|=\omega$, then $\Omega_y:=[y,w_{y,1},\ldots,w_{y,d}]$
is disjoint from $P_1^*$.
\item If $y$ is a vertex of $P^*$ and $x\in \Omega_y$, then
$$
(1+\varepsilon)^{-1}\|y\|^{-d-1}\leq \|x\|^{-d-1}\leq (1+\varepsilon)\|y\|^{-d-1}.
$$
\item  If $y$ is a vertex of $P^*$, then $(1+\varepsilon)\sigma(S^{d-1})^{-1}\|y\|^{-d-1}V(\Omega_y)<1$.
\end{itemize}

Let $y$ be a vertex of $P^*$, which we keep fixed until (\ref{T0}).  
For $p=0,\ldots,d$, and a $(d-1)$-dimensional convex set $F$, we define
$$
\theta^p_{y,F}=\left\{
\begin{array}{ll}
1&\mbox{ if $P$ has a supporting hyperplane at $y$ that is parallel to ${\rm aff}\,F$, and}\\
&\mbox{ aff$\,F$ separates $y$ from $K_1^*$ and
from  $p$ points out of  $w_{y,1},\ldots,w_{y,d}$},\\
0&\mbox{ otherwise. }
\end{array} \right.
$$
We write $\Phi_y$ to denote the linear map with $\Phi_ye_i=w_{y,i}-y$ for $i=1,\ldots,d$, and
hence $\det \Phi_y=d!V(\Omega_y)$. Let
$$
T^p_{q,y}(P^*_n)=\sum_{F\in{\mathcal F}(P^*_n)}\mu^*_P(S_F)^q \theta^p_{y,F}.
$$
A standard argument yields 
\begin{align}
\label{Tp}
\mathbb{E}\,T^p_{q,y}(P^*_n) &= \binom{n}{d} \int_{X_P^d}\mu^*_P(S_{[x_1,\ldots,x_d]})^q\left(1-\mu^*_P(C_{[x_1,\ldots,x_d]})\right)^{n-d} \nonumber\\
&  \hspace*{4mm} \times\, \theta^p_{y,[x_1,\ldots,x_d]} \,\D \mu^*_P(x_1)\cdots \D \mu^*_P(x_d),
\end{align}
where $C_{[x_1,\ldots,x_d]}$ denotes the part of $P^*$ containing $y$ that is cut off by ${\rm aff}\{x_1,\dots,x_d\}$.

We start with the case $p=d$. It follows by the definition 
of $\omega$ that, writing $\beta:=\sigma(S^{d-1})\|y\|^{d+1}$, we have
\begin{align}
\label{ETd1}
\mathbb{E}\,T^d_{q,y}(P^*_n)&\geq \binom{n}{d}(1+\varepsilon)^{-(d+q)}\beta^{-(d+q)}\int_{\Omega_y^d}
V(S_{[x_1,\ldots,x_d]})^q\\
\label{ETd2}
&  \hspace*{4mm} \times\left(1-\frac{(1+\varepsilon)V(C_{[x_1,\ldots,x_d]})}{\beta}\right)^{n-d} 
\theta^d_{y,[x_1,\ldots,x_d]}\,\D x_1\cdots \D x_d.
\end{align}
We write $I^d$ to denote the integral in (\ref{ETd2}).  Applying $\Phi^{-1}$, we deduce that
\begin{align*}
I^d &= \left(d!V(\Omega_y)\right)^{q+d} \int_{\widetilde{\Omega}^d}
V(\widetilde{S}_{[x_1,\ldots,x_d]})^q
\left(1-\frac{d!V(\Omega_y)(1+\varepsilon)V(\widetilde{C}_{[x_1,\ldots,x_d]})}{\beta}\right)^{n-d} \\
&  \hspace*{4mm} \times\,\tilde{\theta}^d_{[x_1,\ldots,x_d]}\,\D x_1\cdots \D x_d.
\end{align*}
Following \cite[p. 298--299]{AW91},  we apply first a Blaschke--Petkantschin formula and the definition of $M_q(\cdot)$, then
(\ref{volwiths}) and (\ref{changofvar}),  to obtain
\begin{align*}
I^d&= \left(d!V(\Omega_y)\right)^{q+d}d^{-q}M_{q+1}(\Delta_{d-1}) (d-1)!\\
&   \hspace*{4mm}\times\, \int_{{\mathcal H}} \tilde{\eta}_H^q\lambda_{d-1}^{d+q+1}(H\cap \widetilde{\Omega})
\left(1-\frac{d!V(\Omega_y)(1+\varepsilon)V(\widetilde{C}_H)}{\beta}\right)^{n-d} 
\tilde{\theta}^d_H\,\D \mu(H)\\
&= \left(d!V(\Omega_y)\right)^{q+d}d^{-q}M_{q+1}(\Delta_{d-1}) (d-1)!^{-(d+q)}\\
& \hspace*{4mm}\times\, \int_{(0,1]^d} (s_1\cdots s_d)^{d+q-1}
\left(1-\frac{V(\Omega_y)(1+\varepsilon)}{\beta}\,s_1\cdots s_d\right)^{n-d} 
\,\D (s_1,\ldots, s_d).
\end{align*}
We may apply (\ref{integral}) because $V(\Omega_y)(1+\varepsilon)/\beta<1$ by the choice of $\omega$.  We deduce by (\ref{ETd1}) that
\begin{align}
\nonumber
\mathbb{E}\,T^d_{q,y}(P^*_n)&\geq \binom{n}{d} (1+\varepsilon)^{-(d+q)}\beta^{-(d+q)}\cdot I^d\\
\label{lowerTd}
&= \frac{(d+q-1)!d^{d-1}M_{q+1}(\Delta_{d-1})}{(d-1)!^2(1+\varepsilon)^{2d+2q}}\,\frac{\ln^{d-1} n}{n^q}
-\Gamma_1  \frac{\ln^{d-2}n}{n^q}.
\end{align}
Now a similar argument with the obvious changes leads to
\begin{equation}
\label{upperTd}
\mathbb{E}\,T^d_{q,y}(P^*_n)\leq 
\frac{(1+\varepsilon)^{2d+2q}(d+q-1)!d^{d-1}M_{q+1}(\Delta_{d-1})}{(d-1)!^2}\,\frac{\ln^{d-1} n}{n^q}
+\Gamma_2  \frac{\ln{d-2}n}{n^q}.
\end{equation}

Next, let $p\in\{1,\ldots,d-1\}$. We define $\tau$ to be the smallest number such that $\tau^{-1}\leq \sigma (S^{d-1}) \|x\|^{d+1}\leq\tau$ for $x\in X_P$, and set $P_y:=\Phi_y^{-1}(P-y)$. Starting from (\ref{Tp}), applying $\Phi_y^{-1}$ and then the Blaschke--Petkantschin formula and (\ref{Mqest}), we obtain
\begin{align*}
&  \mathbb{E}\,T^p_{q,y}(P^*_n)\\
&  \le \Gamma_3n^d \int_{X_P^d}
V(S_{[x_1,\ldots,x_d]})^q\left(1-\frac{V(C_{[x_1,\ldots,x_d]}\cap P^*)}{\tau}\right)^{n-d} 
\theta^p_{y,[x_1,\ldots,x_d]}\,\D x_1\cdots \D x_d\\
&  \le\Gamma_4n^d \int_{P_y^d} V(\widetilde{S}_{[x_1,\ldots,x_d]})^q
\left(1-\frac{d!V(\Omega_y)V(\widetilde{C}_{[x_1,\ldots,x_d]}\cap P_y)}{\tau}\right)^{n-d} 
\tilde{\theta}^p_{[x_1,\dots,x_d]}\\
&  \hspace*{4mm}\times\,\D x_1\cdots \D x_d\\
&  \le\Gamma_5n^d \int_{{\mathcal H}}
\tilde{\eta}_H^q\lambda_{d-1}^{d+q+1}(H\cap P_y)
\left(1-\frac{d!V(\Omega_y)V(\widetilde{C}_H\cap P_y)}{\tau}\right)^{n-d} 
\tilde{\theta}^p_H\,\D\mu(H).
\end{align*}
For $s=(s_1,\ldots,s_d)\in{\mathbb R}_+^d$, we have $\tilde{\theta}^p_{H(s)}=1$ if  exactly $p$ coordinates out of $s_1, \ldots,s_d$ are at most one. In particular, we may assume that $s_i\leq 1$ if $i\leq p$, and $s_i>1$ if $i>p$, and hence
$$
\tilde{\eta}_{H(s)}\lambda_{d-1}(H(s)\cap P_y)\leq\Gamma_6s_1\cdots s_p,
$$
$$
V(\widetilde{C}_{H(s)}\cap P_y) \geq \frac{s_1\cdots s_p}{d!}.
$$
First we change the variable as in (\ref{changofvar}), and secondly we apply (\ref{integral}) or (\ref{integral1}) to obtain in the case $p\in\{1,\ldots,d-1\}$ that
\begin{align}
\nonumber
\mathbb{E}\,T^p_{q,y}(P^*_n)&\leq \Gamma_7n^d
\int_{(1,\infty)^{d-p}}\int_{(0,1]^p}
(s_1\cdots s_p)^{d+q+1}
\left(1-\frac{V(\Omega_y)}{\tau}\, s_1\cdots s_p\right)^{n-d} \\
\nonumber
&  \hspace*{4mm}\times \,s_1^{-2}\cdots s_d^{-2}\,\D (s_1,\ldots, s_p)\,\D(s_{p+1},\ldots, s_d)\\
\label{T1d-1}
&\leq \Gamma_8 \frac{\ln^{p-1} n}{n^q}.
\end{align}
Finally, if $p=0$ then (\ref{Tp}) yields
\begin{equation}
\label{T0}
\mathbb{E}\,T^0_{q,y}(P^*_n)\leq 
\Gamma_9  n^d(1-\mu^*_P(\Omega_y))^{n-d}.
\end{equation}
Since $\mathbb{E}\,T^*_{q}(P^*_n)=\sum_{y\;{\rm vertex}\;{\rm of}\;P^*}\sum_{p=0}^d\mathbb{E}\,T^p_{q,y}(P^*_n)$,
combining (\ref{lowerTd}), (\ref{T1d-1}) and (\ref{T0}) leads to (\ref{lower}),
and combining (\ref{upperTd}), (\ref{T1d-1}) and (\ref{T0}) leads to (\ref{upper}).
This completes the proof of Lemma \ref{lem1}.
\end{proof}

\section{The proofs of Theorem \ref{t3}, (\ref{f0}) and (\ref{fd-1})}
\label{sec4}

In this section, the implied constant in $O(\cdot)$ depends on $P$.
Formula (\ref{f0}) readily follows by Lemma \ref{lem1}, since
$$
\mathbb{E}\,f_0(P^{(n)})=\mathbb{E}\,f_{d-1}(P^*_n)=\mathbb{E}\,T^*_{0}(P^*_n).
$$
Following the proofs of Propositions~1 and 2 in  \cite{AW91}, next we verify

\vspace{3mm}

\begin{lem}\label{lem2}
\begin{equation}
\label{T1V}
\mathbb{E}\,T^*_{1}(P^*_n)\leq \mathbb{E}\,\mu^*_P(P^*\setminus P^*_n)\leq\mathbb{E}\,T^*_{1}(P^*_n)
+O\left( \frac{\ln^{d-2}n}n\right).
\end{equation}
\end{lem}

\vspace{3mm}

{\em Proof.} The lower bound in (\ref{T1V}) is a consequence of the fact that the interiors of the sets $S_F$ are pairwise disjoint as $F$ runs through ${\mathcal F}(P^*_n)$.

The upper bound  in (\ref{T1V}) is proved in several steps. For any convex body $Q$ and $z\in \partial Q$,
let
$$
N(Q,z)=\{u\in{\mathbb R}^d:\,\langle u,x-z\rangle\leq 0\mbox{ for all }x\in Q\}. 
$$
If $Q$ is a polytope and $e$ is an edge of $Q$, then  $N(Q,z)$ is the same $(d-1)$-dimensional cone for any $z$ in the relative interior of $e$, which cone we denote by $N(Q,e)$. In addition, $N(Q,z)$ is $d$-dimensional if $Q=P^*_n$ and $z$ is a proper vertex, or $Q=P^*$ and $z$ is a vertex. For a given $P^*_n$, and for a vertex $y$ and an edge $e$ of $P^*$, we write $f_{0,y}(P^*_n)$ and $f_{0,e}(P^*_n)$ to denote the number of proper vertices $x$ of $P^*_n$ such that $N(P^*_n,x) \subset N(P^*,y)$, respectively that $N(P^*_n,x)$ intersects $N(P^*,e)$. Since for any $z\in P^* \setminus P^*_n$, the cone $N([z,P^*_n],z)$ either is contained
in $N(P^*,y)$ for some vertex $y$ of $P^*$, or intersects
 $N(P^*,e)$ for some edge $e$ of $P^*$, we have
$$
\mathbb{E}\,\mu^*_P(P^*\setminus P^*_n)
\leq\frac1{n+1}\sum_{y\;{\rm vertex}\;{\rm of}\;P^*}\mathbb{E}\,f_{0,y}(P^*_{n+1})
+\frac1{n+1}\sum_{e\;{\rm edge}\;{\rm of}\;P^*}\mathbb{E}\,f_{0,e}(P^*_{n+1}).
$$
Let us consider $x_1,\ldots,x_{n+1}\in X_P$ such that $x_{n+1}$  is a proper vertex of $P^*_{n+1}=[x_1, \ldots, x_{n+1}, P^*_1]$, $N(P^*_{n+1},x_{n+1})\subset N(P^*,y)$ for some vertex $y$ of $P^*$, and $P^*_1\subset {\rm int}\, P^*_n$ for $P^*_{n}=[x_1,\ldots,x_{n},P^*_1]$. In this case, the ray starting from $y$ and passing through $x_{n+1}$ enters into $P^*_{n}$ intersecting a (proper) facet $F$ of $P^*_n$, and $x_{n+1}\in S_F$. Since the probability that $P^*_1\subset {\rm int}\, P^*_n$ is at least $1-O(\gamma^n)$ for some $\gamma\in(0,1)$, we deduce that
\begin{equation}
\label{T1edges}
\mathbb{E}\,\mu^*_P(P^*\setminus P^*_n)
\leq\mathbb{E}\,T^*_{1}(P^*_n)+\frac1{n+1}\sum_{e\;{\rm edge}\;{\rm of}\;P^*}\mathbb{E}\,f_{0,e}(P^*_{n+1})
+O((n+1)\gamma^n).
\end{equation}
Therefore, we fix an edge $e$ of $P^*$ and estimate $\mathbb{E}\,f_{0,e}(P^*_{n+1})$. Let $y$ be one of the endpoints
of $e$. From here on, we use the notation set up in the proof of Lemma \ref{lem1}. We may assume that $w_{y,d}\in e$.
For $x=y+\sum_{i=1}^ds_i(w_{y,i}-y)$ with $s_1,\ldots,s_d\geq 0$,
let
$$
\Xi_x:=[y,w_{y,d},y+\min\{s_1,1\}(w_{y,1}-y),\ldots,y+\min\{s_{d-1},1\}(w_{y,d-1}-y)]
\subset\Omega_y.
$$
In particular, $V(\Xi_x)=\min\{s_i,1\}\cdots \min\{s_{d-1},1\}V(\Omega_y)$. In addition, if $P^*_{n+1} = [x_1,\ldots,x_{n+1}, P^*_1]$ for $x_1,\ldots,x_{n+1}\in X_P$, if $x_{n+1}$ is a proper vertex of $P^*_{n+1}$, and if $N(P^*_{n+1},x_{n+1})$ intersects $N(P^*,e)$, then $\Xi_{x_{n+1}}$ is disjoint from ${\rm int}\,P^*_{n+1}$. 
Considering the number $p$ of the numbers $s_1,\ldots,s_{d-1}$ that
are at most one, we have
\begin{align*}
&  \mathbb{E}\,f_{0,e}(P^*_{n+1})
\le (n+1)\int_{P^*}(1-\mu^*_P(\Xi_x))^n \D \mu^*_P(x)\\
&  \le(n+1)(1-\mu^*_P(\Omega_y))^n
+O(n)\sum_{p=1}^{d-1}\int_{(0,1]^p}\left(1-\frac{V(\Omega_y)}{\tau}\,s_1\cdots s_p\right)^n
\D(s_1,\ldots,s_p).
\end{align*}
Since for $p=1,\ldots,d-1$, the last integral is $O(\frac{\ln^{p-1}n}n)$ according to (\ref{integral}) and (\ref{integral1}),
we conclude that $\mathbb{E}\,f_{0,e}(P^*_{n+1})=O(\ln^{d-2}n)$. Therefore, (\ref{T1edges}) yields Lemma \ref{lem2}.
\hfill \qed

We note that $M_2(\Delta_{d-1})=(d-1)!/d^{d-1}(d+1)^{d-1}$ according to Reed \cite{R74}, and hence 
$$\mathbb{E} \,T^*_{1}(P^*_n)\sim\frac{rd}{(d+1)^{d-1}}\,\frac{\ln^{d-1} n}{n}$$ 
by Lemma \ref{lem1}. Therefore, combining (\ref{polar}) and Lemma \ref{lem2} yields Theorem \ref{t3}. For $x,x_1,\ldots,x_n\in X_P$, the point $x$ is a proper vertex of $[x,x_1,\ldots,x_{n-1},P_1^*]$ if and only if $x \notin [x_1,\ldots,x_{n-1},P_1^*]$. We conclude (see also Efron \cite{E65}) that
$$
\mathbb{E}\,f_{d-1}(P^{(n)})=\mathbb{E}\,f_{0}(P^*_n)=n\mathbb{E}\,\mu^*_P(P^*\setminus P^*_{n-1})
\sim \frac{rd}{(d+1)^{d-1}}\,\ln^{d-1}n
$$
and thus assertion (\ref{fd-1}) holds.

\vspace{3mm}

\small

{\noindent K\'aroly J. B\"or\"oczky,
Alfr\'ed R\'enyi Institute of Mathematics\\
Hungarian Academy of Sciences,
PO Box 127,
H--1364 Budapest, Hungary\\
{\it E-mail address:}
carlos@renyi.hu\\
and\\
Department of Geometry, Roland E\"otv\"os University\\
P\'azm\'any P\'eter s\'et\'any 1/C,
H-1117 Budapest, Hungary\\

\noindent Rolf Schneider,
Mathematisches Institut, Albert-Ludwigs-Universit\"at\\
Eckerstr. 1,
D-79104 Freiburg i.~Br., Germany\\
{\it E-mail address:}  rolf.schneider@math.uni-freiburg.de}

\end{document}